\documentclass[12pt,a4paper]{amsart}
\usepackage[latin1]{inputenc}
\usepackage{graphicx}
\newcommand{\dens}{\operatorname{dens}}
\newcommand{\Id}{\mathrm{Id}}
\usepackage{amssymb, amsmath}
\usepackage{geometry}
\usepackage{enumerate}
\usepackage[colorlinks=true,linkcolor=blue,urlcolor=blue,citecolor=blue]{hyperref}
\usepackage[all]{xy}
\usepackage{tikz-cd}
\usepackage{comment}

\newtheorem{theorem}{Theorem}[section]
\newtheorem{definition}[theorem]{Definition}
\newtheorem{lemma}[theorem]{Lemma}
\newtheorem{corollary}[theorem]{Corollary}
\newtheorem{proposition}[theorem]{Proposition}

\theoremstyle{definition}
\newtheorem{example}[theorem]{Example}
\newtheorem{remark}[theorem]{Remark}

\newcommand{\FBL}[1]{\mathrm{FBL}\langle #1\rangle}  % Using \langle and \rangle
 
\begin{document}
\setcounter{tocdepth}{1}

\title{On the free Banach lattice generated by a lattice}

\author[A. Ben Rjeb]{Asma~Ben Rjeb}
\address{Preparatory institute for engineering studies Nabeul\\
Research Laboratory of Algebra, Topology, Arithmetic, and Order\\
Department of Mathematics\\
Faculty of Mathematical, Physical and Natural Sciences of Tunis\\
Tunis-El Manar University, 2092-El Manar, Tunisia}
\email{asmabenrjab1992@gmail.com }

\author[P. Tradacete]{Pedro~Tradacete}
\address{Instituto de Ciencias Matem\'aticas (CSIC-UAM-UC3M-UCM)\\
Consejo Superior de Investigaciones Cient\'ificas\\
C/ Nicol\'as Cabrera, 13--15, Campus de Cantoblanco UAM\\
28049 Madrid, Spain.}
\email{pedro.tradacete@icmat.es}

\begin{abstract}
We study structural properties of the free Banach lattice $\FBL L$
generated by a distributive lattice $L$. We characterize when $\FBL L$
has a strong unit, compute its density character, analyze the density character
of order intervals and study when is $FVL\langle L\rangle$ order dense in $\FBL L$. We also study projection bands, quasi-interior points,
and Banach lattice homomorphisms induced by lattice homomorphisms. Finally,
we show that $\FBL L$ is lattice isometric to $\FBL L^{\mathrm{op}}$, where $L^{\mathrm{op}}$ denotes the opposite lattice.
\end{abstract}



\maketitle

\section{Introduction}

The construction of free Banach lattices has become a useful tool for
studying the interaction between universal properties, Banach space versus lattice
properties, and order structure. Quite some time after the introduction of free vector lattices in \cite{B1968}, the first appearance of free Banach lattices is due to
de Pagter and Wickstead who introduced the free Banach lattice $\mathrm{FBL}(A)$
associated with a set $A$ \cite{DPW2015}: $\mathrm{FBL}(A)$ is a Banach lattice together with a bounded map $\delta:A\rightarrow \mathrm{FBL}(A)$ such that for every Banach lattice $X$ and bounded map $T:A\rightarrow X$, there is a unique Banach lattice homomorphism $\widehat T:\mathrm{FBL}(A)\rightarrow X$ making the following diagram commutative
\[ \begin{tikzcd}
            FBL(A)  \arrow[rd, "\widehat{T}"] & \\
            A \arrow[r, "T"] \arrow[u, "\delta"] & X
        \end{tikzcd}
\]
and such that $\|\widehat T\|=\sup_{a\in A}\|Ta\|$.

Avil\'es, Rodr\'iguez, and the second author
subsequently constructed the free Banach lattice $\mathrm{FBL}[E]$ generated
by a Banach space $E$ \cite{ART2018}, which generalizes the above construction, allowing to identify $\mathrm{FBL}(A)$ 
lattice isometrically with $\mathrm{FBL}[\ell_1(A)]$. In recent years, these objects have received considerable attention in a number of different research directions, including: projectivity \cite{AMCRA2020}, chain conditions \cite{APRA2018}, convexity \cite{JSLTTT2022}, octahedrality \cite{DMCRARZ2021}, approximation properties \cite{O2024}, complex Banach lattices \cite{DHT2023}, duality \cite{GST2024}, upper $p$-estimates \cite{GSLTT2025}, quasi-Banach spaces \cite{SATTA2025}, $f$-algebras \cite{MLT2025}... We refer to the monograph \cite{OTTT2024} for further developements.
 
It is of course natural to consider how this construction changes when the generating object
already carries an order structure. For a distributive lattice $L$, Avil\'es and
Rodr\'iguez Abell\'an introduced the free Banach lattice $\FBL L$ generated
by $L$ \cite{ARA2019}. This is characterized by a bounded lattice homomorphism
$\delta_L:L\to\FBL L$ with the following universal property: for every
Banach lattice $X$ and every bounded lattice homomorphism $T:L\to X$, there
is a unique Banach lattice homomorphism $\widehat T:\FBL L\to X$ satisfying
$\widehat T\delta_L=T$ and $\|\widehat T\|=\sup_{x\in L}\|Tx\|$. See also \cite{DJJ2026} for free Banach lattices generated by ordered Banach spaces.

Several structural aspects of $\FBL L$ have been studied. In particular, Avil\'es and Rodr\'iguez Abell\'an \cite{ARA2019}  gave a concrete description of $\FBL L$ as a Banach lattice of real-valued functions on the set $L^*$ of all lattice homomorphisms from $L$ to $[-1,1]$. Later, in \cite{AMCRARZ2022b}, they analyzed the behavior of the canonical homomorphism induced by a lattice embedding $i : L \hookrightarrow M$, proving that the induced operator
\[
\overline{i} : \FBL L \longrightarrow \FBL M
\]
is injective if and only if every lattice homomorphism on $L$ extends to $M$, and that it is isometric when $L$ is locally complemented in $M$. They also obtained several projectivity results and the surprising fact that $\FBL L$ is lattice isomorphic to an AM-space (see also \cite{AMCRARZ2022a,ARA2019b}).

The purpose of this paper is to continue the investigation of the structural properties of $\FBL L$ by exploiting its functional representation. Our main contributions are as follows. We first characterize exactly when $\FBL L$ admits a strong unit: this happens if and only if $L$ possesses both a maximum and a minimum (Theorem~\ref{thm:strongunit}). As a consequence, we determine when $\FBL L$ has the strong Nakano property. We then compute the density character of $\FBL L$, showing that $\dens(\FBL L)=\max\{\aleph_0,|L|\}$ (Theorem~\ref{thm:density}); in particular, $\FBL L$ is separable precisely when $L$ is countable. Turning to order intervals, we demonstrate that, unlike the case of $\mathrm{FBL}[E]$, order intervals in $\FBL L$ can have very different density characters. We find a connection between order density of the vector sublattice $FVL\langle L\rangle$ in $\FBL L$ in terms of countable order bounded lattices (Theorem \ref{thm:order dense cob}). We also prove that $\FBL L$ possesses no non-trivial projection bands whenever $|L|>1$ (Theorem \ref{thm:no proj bands}), and we characterize the existence of a quasi-interior point in terms of a countable separating subset of $L$ (Theorem~\ref{thm:quasi-interior-fbl}). 

Furthermore, we study Banach lattice homomorphisms between free Banach lattices over lattices. We show that the canonical extension $\overline{T}:\FBL L\to\FBL M$ of a lattice homomorphism $T:L\to M$ is given by composition with the dual map, and we prove that surjectivity of $T$ implies surjectivity of $\overline{T}$ (Theorem~\ref{thm:onto}). Finally, we show that $\FBL L$ is always lattice isometric to $\FBL{L^{\mathrm{op}}}$ (Proposition~\ref{prop:dual-isometry}). These results provide a comprehensive structural picture of the free Banach lattice generated by a lattice.

\section{Preliminaries}
\subsection{Lattices and lattice homomorphisms}

We recall the order-theoretic notation and the construction of the free
Banach lattice generated by a lattice. All lattices considered below are
assumed to be non-empty. When the functional representation of $\FBL L$ is
used, $L$ is assumed to be distributive.

\begin{definition}[Lattice]
A \emph{lattice} is a set $L$ equipped with a partial order $\le$ such that
every pair $x,y\in L$ admits a supremum
\[
x\vee y=\sup\{x,y\}
\]
and an infimum
\[
x\wedge y=\inf\{x,y\}.
\]

\end{definition}

\begin{definition}[Distributive lattice]
A lattice $(L,\vee,\wedge)$ is called \emph{distributive} if for all
$x,y,z\in L$ one has
\[
x\wedge (y\vee z)=(x\wedge y)\vee (x\wedge z),
\]
and, equivalently,
\[
x\vee (y\wedge z)=(x\vee y)\wedge (x\vee z).
\]

\end{definition}

\begin{definition}[Lattice homomorphism]
Let $L_1$ and $L_2$ be lattices. A map $\phi:L_1\to L_2$ is a
\emph{lattice homomorphism} if
\[
\phi(x\vee y)=\phi(x)\vee\phi(y), \qquad
\phi(x\wedge y)=\phi(x)\wedge\phi(y),
\qquad x,y\in L_1.
\]
\end{definition}

For a subset $A$ of a Banach lattice, we denote by $lat(A)$ the smallest
vector sublattice containing $A$, that is, the set of all lattice-linear
combinations of elements of $A$.

\subsection{Construction of $\FBL L$}

Let $L$ be a distributive lattice. We recall the main constructions of the
free Banach lattice generated by $L$, introduced in \cite{ARA2019} and further
developed in \cite{AMCRARZ2022a,AMCRARZ2022b}.

\begin{definition}[Free Banach lattice generated by a lattice]
The \emph{free Banach lattice generated by $L$} is a Banach lattice $\FBL L$
together with a bounded lattice homomorphism $\delta_L:L\to\FBL L$ satisfying
the following universal property: for every Banach lattice $X$ and every
bounded lattice homomorphism $T:L\to X$, there exists a unique Banach lattice
homomorphism $\widehat T:\FBL L\to X$ such that
$\widehat T\circ\delta_L=T$ and
$\|\widehat T\|=\sup_{x\in L}\|Tx\|$.
\end{definition}

The universal property determines $\FBL L$ uniquely up to canonical lattice
isometric isomorphism; see \cite[Section~2]{ARA2019}.

Let $L^*$ denote the set of all lattice homomorphisms 
$x^* : L \to [-1,1]$.  
For each $x\in L$ define
\[
\delta_x(x^*)=x^*(x), \qquad x^*\in L^*.
\]
A function $f:L^*\to\mathbb R$ is called \emph{positively homogeneous} if
$f(\lambda x^*)=\lambda f(x^*)$ for every $x^*\in L^*$ and every
$\lambda\ge0$ such that $\lambda x^*\in L^*$, where scalar multiplication is
understood pointwise. For such a function, let
\begin{equation}\label{eq:FBL-L-norm}
\|f\|
=\sup\left\{
\sum_{i=1}^m |f(x^*_i)|:
m\in\mathbb N,\ x^*_1,\dots,x^*_m\in L^*,\
\sup_{x\in L}\sum_{i=1}^m |x^*_i(x)|\le 1
\right\}.
\end{equation}

For $x\in L$, the evaluation map $\delta_x:L^*\to\mathbb R$ given by
$\delta_x(x^*)=x^*(x)$ is positively homogeneous and satisfies
$\|\delta_x\|\le 1$.

Consider the Banach lattice generated by $\{\delta_x:x\in L\}$ inside
the space of all positively homogeneous functions $f\in\mathbb R^{L^*}$ with
$\|f\|<\infty$, endowed with this norm and the pointwise lattice operations.
Together with the assignment $\delta_L(x)=\delta_x$, this Banach lattice has
the universal property above. Thus $\FBL L$ can be represented as a Banach
lattice of functions on the dual object $L^*$.

We also recall an alternative description using a quotient of a free Banach lattice generated by a set: Forget the lattice structure of
$L$ temporarily and regard it as a set. Let $\mathrm{FBL}(L)$ be the free
Banach lattice generated by this set, with canonical map
$u:L\to\mathrm{FBL}(L)$, as in \cite{DPW2015}. Let $I$ be the closed ideal of
$\mathrm{FBL}(L)$ generated by
\[
u(x)\vee u(y)-u(x\vee y), \qquad
u(x)\wedge u(y)-u(x\wedge y),
\qquad x,y\in L .
\]
These relations force the image of $L$ to respect the lattice operations. If
$q:\mathrm{FBL}(L)\to\mathrm{FBL}(L)/I$ denotes the quotient map, then
$\delta_L:=q\circ u$ is a lattice homomorphism, and
$\mathrm{FBL}(L)/I$ has the universal property defining $\FBL L$.

It should be noted that although $\FBL L$ can be actually defined even when $L$ is not distributive, in that case, the $\delta$ map is not injective. In fact, it is shown in \cite{ARA2019} that one can replace $L$ with $L'=\delta(L)$ and identify $\FBL L$ with $\FBL{L'}$. Thus, for our purposes and throughout this paper we will assume all lattices are distributive unless stated otherwise.

\section{Strong units in $\FBL L$}

By \cite[Lemma~2.5]{AMCRARZ2022a}, if $L$ has a maximum $M_L$ and a minimum $m_L$,
then $|\delta_{m_L}|\vee|\delta_{M_L}|$ is a strong unit for $\FBL L$.
Moreover, \cite[Theorem~2.7]{AMCRARZ2022a} identifies $\FBL L$, up to an equivalent
norm, with $C(K_L)$, where
\[
K_L=\{x^*\in L^*:\max\{|x^*(m_L)|,|x^*(M_L)|\}=1\}.
\]
We prove that the existence of both a maximum and a minimum is also necessary
for $\FBL L$ to have a strong unit. This is analogous to
\cite[Proposition~9.1]{OTTT2024}, where $\mathrm{FBL}[E]$ is shown to have a
strong unit precisely when $E$ is finite-dimensional.

\begin{lemma}\label{lem:strong-separation}
Let $L$ be a distributive lattice and let $y\nleq x$ in $L$. Then there exists
$x^*\in L^*$ such that $x^*(x)=0$, $x^*(y)=1$, and
$0\le x^*(z)\le 1$ for every $z\in L$.
\end{lemma}

\begin{proof}
Since $L$ is distributive, \cite[Proposition~3.1]{ARA2019} implies that
$\delta:L\to\FBL L$ is injective. Hence there is $x_0^*\in L^*$ such that
$x_0^*(x)<x_0^*(y)$. For $z\in L$, define
\[
y^*(z)=x_0^*(x)\vee\bigl(x_0^*(z)\wedge x_0^*(y)\bigr),
\]
and set
\[
x^*(z)=\frac{y^*(z)-y^*(x)}{y^*(y)-y^*(x)}.
\]
Then $x^*$ is a lattice homomorphism satisfying the required conditions.
\end{proof}

\begin{theorem}\label{thm:strongunit}
Let $L$ be a lattice. Then $\FBL L$ has a strong unit if and only if $L$
has a maximum and a minimum.
\end{theorem}

\begin{proof}
By \cite[Proposition~3.2]{ARA2019}, we may assume without loss of generality that
$L$ is distributive. If $L$ has a maximum and a minimum, the result follows
from \cite[Lemma~2.5]{AMCRARZ2022a}.

Conversely, suppose that $f\in\FBL L$ is a strong unit. Then, for every
$x\in L$, we have $|\delta_x|\le f$. Let
$g\in lat\{\delta_x:x\in L\}$. Then $g$ is a lattice-linear combination of
$\delta_{x_1},\dots,\delta_{x_n}$ for some $x_1,\dots,x_n\in L$.

Suppose first that $L$ has no maximum. Then there exists
$y\in L$ such that $y>x_1\vee\cdots\vee x_n$. Applying
Lemma~\ref{lem:strong-separation} with
$x=x_1\vee\cdots\vee x_n$, we obtain $z^*\in L^*$ such that
$z^*(y)=1$, $z^*(x)=0$, and $0\le z^*(z)\le 1$ for all $z\in L$.
In particular, $z^*(x_i)=0$ for every $i$, and therefore $g(z^*)=0$.
Hence
\[
\|f-g\| \ge |f(z^*)-g(z^*)| = |f(z^*)| \ge |\delta_y(z^*)| = 1.
\]
Thus $f$ cannot be approximated by elements of $lat\{\delta_x:x\in L\}$,
contradicting the definition of $\FBL L$ as the closed Banach lattice
generated by the functions $\delta_x$.

If $L$ has no minimum, we apply the same argument to the opposite lattice
$L^{\mathrm{op}}$. Then $L^{\mathrm{op}}$ has no maximum, and
Proposition~\ref{prop:dual-isometry} shows that $\FBL L$ is lattice
isometric to $\FBL{L^{\mathrm{op}}}$. Thus a strong unit in $\FBL L$ would
yield a strong unit in $\FBL{L^{\mathrm{op}}}$, contradicting the previous
paragraph. Hence $L$ has both a maximum and a minimum.
\end{proof}

Recall that a Banach lattice $X$ has the \emph{strong Nakano property} if
every upward directed norm-bounded subset of $X$ has a supremum. The strong Nakano property was studied in the context of free Banach lattices generated by a Banach space in \cite{ABRT2025}. For
AM-spaces this property is equivalent to the existence of a strong unit.
The following corollary is therefore an immediate consequence of
Theorem~\ref{thm:strongunit} and the fact that $\FBL L$ is lattice
isomorphic to an AM-space \cite[Theorem~3.9]{AMCRARZ2022b}.

\begin{corollary}
$\FBL L$ has the strong Nakano property if and only if $L$ has a maximum and a minimum.
\end{corollary}

\begin{proof}
An AM-space has the strong Nakano property if and only if it has a strong
unit \cite[Proposition~2.1]{ABRT2025}. Since $\FBL L$ is lattice isomorphic to an
AM-space by \cite[Theorem~3.9]{AMCRARZ2022b}, the result follows from
Theorem~\ref{thm:strongunit}.
\end{proof}

\section{Density character in $\FBL L$}

Recall that the density character of a topological space $S$, denoted by
$\dens(S)$, is the least cardinality of a dense subset. For a Banach space
$E$, it is easy to check that $\dens(\mathrm{FBL}[E])=\dens(E)$
\cite[Section~3]{ART2018}. In this section we study the corresponding question
for $\FBL L$ and then discuss the density character of order intervals.

\begin{theorem}\label{thm:density}
For every non-empty lattice $L$, we have
\[
\dens(\FBL L)=\max\{\aleph_0,|L|\}.
\]
\end{theorem}

\begin{proof}
By \cite[Section~3]{ARA2019}, we may assume without loss of generality that $L$
is distributive. Given $x\in L$, the line
$\{t\delta_x:t\in\mathbb R\}$ is contained in $\FBL L$, and hence
$\dens(\FBL L)\ge\aleph_0$.

Also, if $x\ne y$ in $L$, then, after interchanging $x$ and $y$ if
necessary, we may assume that $x\wedge y<x$. By
Lemma~\ref{lem:strong-separation}, there exists $x^*\in L^*$ such that
$x^*(x\wedge y)=0$, $x^*(x)=1$, and $0\le x^*(z)\le 1$ for every $z\in L$.
Since $x^*$ is a lattice homomorphism, this gives $x^*(y)=0$, and therefore
\[
\|\delta_x - \delta_y\| \geq |x^*(x) - x^*(y)| = 1 > 0.
\]
Thus $\{\delta_x:x\in L\}$ is a $1$-separated subset of $\FBL L$, and so
\[
\dens(\FBL L)\ge |L|.
\]

Conversely, $\FBL L$ is the quotient of $\mathrm{FBL}(L)$ by the closed ideal
generated by
\[
\{u(x)\vee u(y)-u(x\vee y),\ u(x)\wedge u(y)-u(x\wedge y):x,y\in L\},
\]
where $u:L\to\mathrm{FBL}(L)$ is the canonical embedding; see
\cite[Section~2]{ARA2019}. Hence
\[
\dens(\FBL L)\le \dens(\mathrm{FBL}(L))
=\max\{\aleph_0,|L|\}.
\]
The equality $\dens(\mathrm{FBL}(L))=\max\{\aleph_0,|L|\}$ follows from
\cite[Section~8]{DPW2015}.
\end{proof}

\begin{remark}
In particular, $\FBL L$ is separable if and only if $L$ is countable.
\end{remark}

In the free Banach lattice generated by a Banach space, all order intervals
have the same density character \cite[Section~3]{ART2018}. As we will see, the
situation in $\FBL L$ is more involved. For $a\le b$ in $L$, write
\[
[a,b]_L=\{x\in L:a\le x\le b\}.
\]

\begin{proposition}\label{prop:interval-density}
Let $L$ be a distributive lattice and let $a,b\in L$ with $a<b$. Then
\[
\max\{\aleph_0,|[a,b]_L|\}\le
\dens\big([\delta(a),\delta(b)]_{\FBL L}\big)
\le \max\{\aleph_0,|L|\}.
\]
\end{proposition}

\begin{proof}
Since $\delta(a)+t(\delta(b)-\delta(a))\in[\delta(a),\delta(b)]_{\FBL L}$
for every $t\in[0,1]$, we have
$\dens([\delta(a),\delta(b)]_{\FBL L})\ge\aleph_0$. Also, for distinct
$x,y\in[a,b]_L$, Lemma~\ref{lem:strong-separation} gives
$x^*\in L^*$ such that $x^*(x)=0$ and $x^*(y)=1$, or vice versa. Hence
$$
\|\delta(x)-\delta(y)\|\geq|x^*(x)-x^*(y)|\geq1.
$$
Since $\delta(x),\delta(y)\in[\delta(a),\delta(b)]_{\FBL L}$, the set
$\{\delta(x):x\in[a,b]_L\}$ is $1$-separated in the interval. Therefore
\[
\dens\big([\delta(a),\delta(b)]_{\FBL L}\big)\ge |[a,b]_L|.
\]

The upper estimate follows from
\[
\dens\big([\delta(a),\delta(b)]_{\FBL L}\big)\le\dens(\FBL L)
=\max\{\aleph_0,|L|\},
\]
using Theorem~\ref{thm:density}.
\end{proof}

The preceding estimate does not imply that every interval must have large
density when $|L|$ is large, as the following example shows.

\begin{example}\label{ex:separable interval}
Let $\Gamma$ be a set and let $L=\mathcal P(\Gamma)$, with
$\wedge=\cap$ and $\vee=\cup$. Fix $\gamma_0\in\Gamma$ and put
$a=\emptyset$ and $b=\{\gamma_0\}$. We claim that
$[\delta(a),\delta(b)]_{\FBL L}$ is separable.

Define
\[
U=\{x^*\in L^*:x^*(b)>x^*(a)\},
\]
and consider the space $X_U$ of real-valued continuous functions on $U$,
equipped with the norm
\[
\|f\|=\sup\left\{\sum_{i=1}^m |f(x^*_i)|:
m\in\mathbb N,\ x^*_1,\dots,x^*_m\in U,\
\sup_{x\in L}\sum_{i=1}^m |x^*_i(x)|\le 1
\right\}.
\]

Let $r:\FBL L\to X_U$ be the restriction map to $U$. Then $r$ is a lattice
homomorphism and is isometric on $[\delta(a),\delta(b)]$: if
$f\in[\delta(a),\delta(b)]$ and $x^*\notin U$, then
$f(x^*)=x^*(a)=x^*(b)$.

For $A\in\mathcal P(\Gamma)$, if $\gamma_0\notin A$, then
$r(\delta(A))=r(\delta(a))$, while if $\gamma_0\in A$, then
$r(\delta(A))=r(\delta(b))$.

Thus $r[\delta(a),\delta(b)]$ is contained in the separable sublattice
generated by $r(\delta(a))$ and $r(\delta(b))$. Since $r$ is isometric on the
interval, $[\delta(a),\delta(b)]$ is separable.
\end{example}

\begin{remark}\label{rem:interval-density-character}
Thus order intervals in $\FBL L$ need not have the same density character.
Indeed, let $\Gamma$ be infinite and let $L=\mathcal P(\Gamma)$ as above.
Then
\[
\dens([\delta(\emptyset),\delta(\{\gamma_0\})])=\aleph_0.
\]
On the other hand, for every infinite $A\subset\Gamma$,
\[
\dens([\delta(\emptyset),\delta(A)])\ge |[\emptyset,A]_L|=2^{|A|}>\aleph_0.
\]
\end{remark}

\begin{remark}
In general, there is no upper bound on $\dens([\delta(a),\delta(b)])$
depending only on $|[a,b]_L|$. To see this, fix an infinite cardinal
$\kappa$ and choose a set $\Gamma$ with $|\Gamma|=\kappa$.
Let
\[
L=\{\hat 0\}\cup\mathcal P(\Gamma),
\]
ordered so that $\hat 0$ is a new minimum below every element of
$\mathcal P(\Gamma)$, while on $\mathcal P(\Gamma)$ we use the usual
operations $\wedge=\cap$ and $\vee=\cup$. Then $L$ is a distributive lattice.

Define
\[
a=\hat 0,\qquad b=\emptyset\in\mathcal P(\Gamma).
\]
Then $a<b$ and
\[
|[a,b]_L|=|\{\hat 0,\emptyset\}|=2.
\]

For each $\gamma\in\Gamma$, define $x^*_\gamma\in L^*$ by
\[
x^*_\gamma(\hat 0)=-1,\qquad
x^*_\gamma(A)=\chi_A(\gamma)\quad(A\subseteq\Gamma).
\]

For each $\gamma\in\Gamma$, consider
\[
f_\gamma=
\bigl((\delta(a)+\delta(\{\gamma\}))\vee\delta(a)\bigr)\wedge\delta(b)
\in[\delta(a),\delta(b)]\subset\FBL L.
\]
Note that
\[
\delta(a)(x^*_\gamma)=x^*_\gamma(\hat 0)=-1,\qquad \delta(b)(x^*_\gamma)=x^*_\gamma(\emptyset)=0,
\]
and
\[
\delta(\{\gamma\})(x^*_\gamma)=x^*_\gamma(\{\gamma\})=1,\qquad
\delta(\{\gamma'\})(x^*_\gamma)=x^*_\gamma(\{\gamma'\})=0\quad(\gamma'\neq\gamma).
\]
Therefore,
\[
f_\gamma(x^*_\gamma)=\min\{\max\{0,-1\},0\}=0,\qquad
f_{\gamma'}(x^*_\gamma)=\min\{\max\{-1,-1\},0\}=-1\quad(\gamma'\neq\gamma),
\]
so
\[
|(f_\gamma-f_{\gamma'})(x^*_\gamma)|=1\qquad(\gamma\neq\gamma').
\]
Therefore
\[
\|f_\gamma-f_{\gamma'}\|\ge
\sup_{x^*\in L^*}|(f_\gamma-f_{\gamma'})(x^*)|
\ge |(f_\gamma-f_{\gamma'})(x^*_\gamma)|=1.
\]
Thus $\{f_\gamma:\gamma\in\Gamma\}\subseteq[\delta(a),\delta(b)]$ is a
$1$-separated set of cardinality $\kappa$, and hence
\[
\dens([\delta(a),\delta(b)])\ge\kappa.
\]
Since $\kappa$ was arbitrary while $|[a,b]_L|=2$, no upper bound on
$\dens([\delta(a),\delta(b)])$ can depend only on $|[a,b]_L|$.
\end{remark}

Recall that in the setting of the free Banach lattice generated by a Banach space, $FVL[E]$, the vector lattice generated by $\delta(E)$, is order dense in $FBL[E]$ if and only if $E$ is finite dimensional (see \cite{AD2025,OTTT2024}). Our next aim is to study which lattices $L$ satisfy that $\mathrm{FVL}\langle L\rangle$, the vector lattice generated by $\delta(L)$, is order dense in $\FBL L$. We need some preparatory lemmas first.

\begin{lemma}\label{lem:order-dense-aux}
Let $L$ be a lattice with minimum $0_L$ and maximum $1_L$, and let
\[
K_L=\{x^*\in L^*: \max\{|x^*(0_L)|,|x^*(1_L)|\}=1\}
\]
be endowed with the product topology. If $U\subseteq K_L$ is open,
$x_0^*\in U$, and $\varepsilon>0$, then there exists
$g\in\mathrm{FVL}\langle L\rangle_+$ such that $g(x_0^*)>0$,
$g\le\varepsilon$ on $U$, and $g=0$ on $K_L\setminus U$.
\end{lemma}

\begin{proof}
Recall that $\mathrm{FBL}\langle L\rangle$ is lattice isomorphic to $C(K_L)$
\cite[Section~2]{AMCRARZ2022a}. Under this isomorphism,
$\mathrm{FVL}\langle L\rangle$ corresponds to a dense sublattice of $C(K_L)$
containing the constant functions. Indeed, the constant function $1$
corresponds to $|\delta_{0_L}|\vee|\delta_{1_L}|$.

Let $U\subseteq K_L$ be open, let $x_0^*\in U$, and let $\varepsilon>0$.
Since $K_L$ is compact Hausdorff, Urysohn's lemma gives
$v\in C(K_L)$ such that
\[
0\le v\le 1,\qquad v(x_0^*)=1,\qquad v=0\text{ on }K_L\setminus U.
\]

Choose $u\in\mathrm{FVL}\langle L\rangle$ such that
$\|v-u\|_\infty<1/3$.

Set $w=(u-1/3)^+$. Then $w\in\mathrm{FVL}\langle L\rangle_+$.
If $x^*\in K_L\setminus U$, then $v(x^*)=0$ and $u(x^*)<1/3$, so
$w(x^*)=0$. Moreover, $u(x_0^*)>2/3$, and hence $w(x_0^*)>1/3$.
Finally, since $u\le v+1/3\le 4/3$, we have $w\le 1$.

Set $g=\varepsilon w$. Then $g\in\mathrm{FVL}\langle L\rangle_+$,
$g(x_0^*)>0$, $g=0$ on $K_L\setminus U$, and $g\le\varepsilon$ on $U$.
\end{proof}

\begin{lemma}\label{lem:order-dense}
Let $L$ be a distributive lattice with a maximum and a minimum. Then
$\mathrm{FVL}\langle L\rangle$ is order dense in
$\mathrm{FBL}\langle L\rangle$.
\end{lemma}

\begin{proof}
Let $0_L$ and $1_L$ denote the minimum and maximum of $L$, respectively.
Recall that $\mathrm{FBL}\langle L\rangle$ is lattice isomorphic to
$C(K_L)$ \cite[Section~2]{AMCRARZ2022a}, where
\[
K_L=\{x^*\in L^*: \max\{|x^*(0_L)|,|x^*(1_L)|\}=1\}
\]
is endowed with the product topology.

Let $0<f\in\mathrm{FBL}\langle L\rangle$. Choose $x_0^*\in K_L$ such that
$f(x_0^*)>0$. By continuity, there are an open neighbourhood
$U\subseteq K_L$ of $x_0^*$ and $\varepsilon>0$ such that
$f(x^*)>\varepsilon$ for all $x^*\in U$.

By Lemma~\ref{lem:order-dense-aux}, there exists
$g\in\mathrm{FVL}\langle L\rangle_+$ such that $g(x_0^*)>0$,
$g\le\varepsilon$ on $U$, and $g=0$ on $K_L\setminus U$.

If $x^*\in U$, then $g(x^*)\le\varepsilon<f(x^*)$, while if
$x^*\in K_L\setminus U$, then $g(x^*)=0\le f(x^*)$. Hence
$0<g\le f$. Since $f>0$ was arbitrary,
$\mathrm{FVL}\langle L\rangle$ is order dense in
$\mathrm{FBL}\langle L\rangle$.
\end{proof}

Let us say that a lattice $L$ is \emph{countably order bounded} if 
for every countable subset $A\subseteq L$ there are $a,b\in L$ such that
\[
        a\leq x\leq b \qquad (x\in A).
\]

\begin{theorem}\label{thm:order dense cob}
Let $L$ be a distributive lattice. If $L$ is countably order bounded, then $FVL\langle L\rangle$ is order dense in $\FBL L$. If $L$ is totally ordered, the converse also holds.
\end{theorem}

\begin{proof}
Suppose first that $L$ is countably order bounded and let
\[
        0<f\in \FBL L_+ .
\]
Since $f$ depends on a countable set $A\subseteq L$, there are $a,b\in L$ such that
\[
        a\leq x\leq b \qquad (x\in A).
\]
Let \(I=[a,b]_{L}\). This interval is a bounded distributive lattice, with
minimum $a$ and maximum $b$, and it contains $A$. The maps
\[
\begin{array}{cclcc}
         i:I\longrightarrow L,& &i(x)=x, &&\text{ and }\\
         r:L\longrightarrow I,& &r(x)=a\vee(x\wedge b),&&
\end{array}
\]
are lattice homomorphism with $ri(x)=x$ for $x\in I$. The universal property allows to extend these to Banach lattice homomorphisms $\overline{i}:\FBL I\to \FBL L$ and $\overline{r}:\FBL L\to \FBL I$ (see Section \ref{sec:lattice homo FBL}). By \cite[Proposition 3.1]{AMCRARZ2022b}, $\overline{i}:\FBL I\to \FBL L$ is an isometric embedding, so we can consider $\FBL I$ as a complemented sublattice of $\FBL L$ (the projection being the composition $\overline{i}\overline{r}$.)

By Lemma \ref{lem:order-dense}, there is $g\in FVL\langle I\rangle$ such that $0<g\leq \overline{r}(f)$.
Since $f$ depends only on the coordinates in $A\subseteq I$, we have that 
\[
0<\overline{i}g\leq \overline{i}\overline{r}f=f.
\]
Since $\overline{i}g\in FVL\langle L\rangle$, it follows that $FVL\langle L\rangle$ is order dense in $\FBL L$.

For the converse, we assume that $L$ is totally ordered. Suppose first that $L$
has a countable subset with no common upper bound. Replacing it by the
sequence of its finite suprema, we have an increasing sequence
$(u_n)_{n\geq0}$ with no upper bound. Fix positive scalars
$\alpha_n>0$ with $\sum_{n=1}^{\infty}\alpha_n<1$, and define
\[
          f(x^*)=
          \left(|\delta_{u_0}|(x^*)-\sum_{n=1}^{\infty}\alpha_n
          |\delta_{u_n}|(x^*)\right)_+ ,
          \qquad x^*\in L^* .
\]
The series converges uniformly, so $f\in \FBL{L}_+$. Moreover,
the constant homomorphism $\mathbf 1:L\to[-1,1]$, $\mathbf 1(x)=1$, gives
\[
          f(\mathbf 1)=1-\sum_{n=1}^{\infty}\alpha_n>0,
\]
so $f\neq0$.

Let $0\leq h\in FVL(L)$ and suppose $h\leq f$. The function $h$
is in the lattice generated by some finite set $\{\delta_{x_1},\dots,\delta_{x_n}\}$. Let $F$ be the finite sublattice of $L$ generated by $\{u_0,x_1,\dots, x_n\}$, and let $v=\bigvee F$. 

Since $(u_n)$ has no upper bound, and $L$ is totally ordered, choose $n$ with $v<u_n$. Define the 
map
\[
          y^*(x)=
          \begin{cases}
          0, & x<u_n,\\
          1, & x\ge u_n .
          \end{cases}
\]
Since $L$ is totally ordered, $y^*\in L^*$. Clearly, $y^*$ vanishes on $F$ and $y^*(u_n)=1$. 

If $h\neq0$, then we can choose $x^*\in L^*$ with $h(x^*)>0$. For $0<t\le1$, put
\[
          z_t^*=t x^*+(1-t)y^* .
\]
Since $L$ is totally ordered, $L^*$ is closed under convex combinations, so
$z_t^*\in L^*$.  On the coordinates on which $h$ depends, we have that
$z_t^*=t x^*$. Hence,
\[
          h(z_t^*)=t\,h(x^*)>0.
\]
On the other hand,
\[
          f(z_t^*)\le
          \left(t|x^*(u_0)|
          -\alpha_n|tx^*(u_n)+(1-t)|\right)_+ .
\]
For sufficiently small $t>0$, the expression inside the positive part is
negative, and therefore $f(z_t^*)=0$. This contradicts $0\leq h\leq f$.
Hence no non-zero element of $FVL\langle L\rangle_+$ is dominated by $f$, so
order density fails.

The case where $L$ has a countable subset with no common lower bound is
similar.
\end{proof}

We do not know whether the hypothesis of $L$ being totally ordered can be removed completely from the previous result.

\section{$\FBL L$ has no non-trivial projection bands}

The study of projection bands in free Banach lattices goes back to de Pagter
and Wickstead \cite{DPW2015} for the case of a generating set, and was later
observed to extend to free Banach lattices generated by Banach spaces in \cite[Section~9]{OTTT2024}. In this section, we prove the corresponding result for free Banach lattices generated by
lattices: when $|L|>1$, the only projection bands in $\FBL L$ are $\{0\}$
and $\FBL L$. We then derive the corresponding consequences for order
completeness, order continuity of the norm, and atoms.

\begin{lemma}
For a distributive lattice $L$, the set $L^*$ contains a non-constant
element if and only if $|L|>1$.
\end{lemma}

\begin{proof}
It is enough to prove that if $|L|>1$, then $L^*$ contains a non-constant
element. Take $x\ne y$ in $L$. Replacing $x,y$ by $x\wedge y$ and
$x\vee y$, we may assume that $x<y$. Let
\[
J=\{z\in L:z\le x\},\qquad G=\{z\in L:z\ge y\}.
\]
Then $J$ is an ideal, $G$ is a filter, and $J\cap G=\emptyset$. By
\cite[Theorem~10.21]{DP2002}, there is a prime ideal $I$ such that
$J\subset I$ and $I\cap G=\emptyset$. Hence we can find $x^*\in L^*$ such
that $x^*(z)=0$ precisely when $z\in I$. This $x^*$ is not constant.
\end{proof}

\begin{lemma}
If $|L|>1$, then $L^*\setminus\{0\}$ is path connected.
\end{lemma}

\begin{proof}
Let $u\in L^*$ be non-constant, as in the preceding lemma. Hence, the affine paths
$t\mapsto (1-t)u+t\mathbf 1$ and $t\mapsto (1-t)u-t\mathbf 1$ do not pass
through the zero function; otherwise $u$ would be constant. Thus every
non-constant element is connected to both $\mathbf 1$ and $-\mathbf 1$.
A non-zero constant homomorphism is connected to either $\mathbf 1$ or
$-\mathbf 1$ through non-zero constant homomorphisms. Hence every element
of $L^*\setminus\{0\}$ lies in the same path component.
\end{proof}

\begin{definition}
Let $X$ be a non-empty set and let $f:X\to\mathbb R$. Define
\[
O_f=\{x\in X:f(x)\ne 0\}.
\]
If $W$ is a non-empty subset of $\mathbb R^X$, define
\[
O_W=\bigcup_{f\in W}O_f.
\]
\end{definition}

\begin{theorem}\label{thm:no proj bands}
If $L$ is a distributive lattice with $|L|\ge 2$, then the only projection
bands of $\mathrm{FBL}\langle L\rangle$ are $\{0\}$ and
$\mathrm{FBL}\langle L\rangle$.
\end{theorem}

\begin{proof}
By \cite[Proposition~7.4]{BGSDHMCT2026}, we may identify
$\mathrm{FBL}\langle L\rangle$ with the space $C_{ph}(L^*)$ of continuous
positively homogeneous functions on $L^*$. By \cite[Proposition~6.1]{DPW2015},
it is enough to show that $O_{\mathrm{FBL}\langle L\rangle}$ is connected.
For $l\in L$, write
\[
O_{\delta_l}=\{\xi\in L^*:\xi(l)\ne 0\}.
\]
Since
\[
\bigcup_{l\in L}O_{\delta_l}=L^*\setminus\{0\}
\]
and every positively homogeneous function vanishes at $0$, we obtain
\[
O_{\mathrm{FBL}\langle L\rangle}=L^*\setminus\{0\}.
\]
By the preceding lemma, this set is connected. Therefore the only projection
bands are $\{0\}$ and the whole space.
\end{proof}

\begin{corollary}
If $|L|\ge 2$, then:
\begin{enumerate}
\item $\mathrm{FBL}\langle L\rangle$ is not Dedekind $\sigma$-complete.
\item The norm on $\mathrm{FBL}\langle L\rangle$ is not order continuous.
\item $\mathrm{FBL}\langle L\rangle$ has no atoms.
\end{enumerate}
\end{corollary}

\begin{proof} 
\begin{enumerate}
\item This follows from \cite[Proposition~1.2.11]{MN1991}.
\item This follows from \cite[Theorem~2.4.2]{MN1991}.
\item The closed linear span of an atom generates a projection band. Since
there are no non-trivial projection bands, no atoms can exist.
\end{enumerate}
\end{proof}

\begin{corollary}
If $l\in L$, then $|\delta_l|$ is a weak order unit of
$\mathrm{FBL}\langle L\rangle$.
\end{corollary}

\begin{proof}
Suppose that $f\in\mathrm{FBL}\langle L\rangle$ and
$f\perp|\delta_l|$. Then
\[
O_f\subset\{\xi\in L^*:\xi(l)=0\}.
\]
The set on the right has empty interior, so $O_f=\emptyset$ and hence
$f=0$.
\end{proof}

\section{Quasi-interior points in $\mathrm{FBL}\langle L\rangle$}

For a Banach lattice $E$ and $u\in E_+$, recall that $u$ is a
\emph{quasi-interior point} if the principal ideal
\[
E_u=\{x\in E:|x|\le cu\text{ for some }c\ge0\}
\]
is dense in $E$.

For free Banach lattices generated by Banach spaces, $\mathrm{FBL}[E]$ has
a quasi-interior point if and only if $E$ is separable
\cite[Theorem~10.4]{OTTT2024}. In the lattice setting, the
corresponding condition is the existence of a countable separating subset.

\begin{theorem}\label{thm:quasi-interior-fbl}
Let $L$ be a distributive lattice. Then $\mathrm{FBL}\langle L\rangle$ has a
quasi-interior point if and only if $L$ contains a countable separating
subset, that is, a countable set $S\subset L$ such that for every non-zero
$\phi\in L^*$ there exists $s\in S$ with $\phi(s)\ne0$.
\end{theorem}

\begin{proof}
Suppose first that $u\in\mathrm{FBL}\langle L\rangle_+$ is a
quasi-interior point. Since $lat\{\delta_x:x\in L\}$ is dense in
$\mathrm{FBL}\langle L\rangle$, there is a sequence $(f_n)$ in
$\mathrm{FBL}\langle L\rangle$ such that:
\begin{enumerate}[(i)]
\item $f_n\to u$ in norm;
\item each $f_n$ belongs to the vector sublattice generated by
$\{\delta_x:x\in F_n\}$ for some finite set $F_n\subset L$.
\end{enumerate}

Set $S=\bigcup_{n\ge1}F_n$. Then $S$ is countable.

Let $0\ne\phi\in L^*$. Choose $x\in L$ with $\phi(x)\ne0$. Since $u$ is
quasi-interior, the ideal $\mathrm{FBL}\langle L\rangle_u$ is dense in
$\mathrm{FBL}\langle L\rangle$. Hence there is a sequence $(g_k)$ in this
ideal such that $g_k\to\delta_x$ in norm. Evaluating at $\phi$, we get
$g_k(\phi)\to\phi(x)\ne0$, so $g_k(\phi)\ne0$ for some $k$. Since
$|g_k|\le c_k u$ for some $c_k\ge0$, it follows that $u(\phi)>0$.

Since $f_n\to u$ in norm, and evaluation at $\phi$ is continuous, we have
$f_n(\phi)>0$ for all sufficiently large $n$. Because $f_n$ belongs to the
vector sublattice generated by $\{\delta_s:s\in F_n\}$, this forces
$\phi(s)\ne0$ for some $s\in F_n\subset S$. Thus $S$ is separating.

Conversely, let $S=\{s_n:n\ge1\}$ be a countable separating subset. Define
\[
u=\sum_{n=1}^{\infty}2^{-n}|\delta_{s_n}|
\in\mathrm{FBL}\langle L\rangle_+.
\]
The series converges in $\mathrm{FBL}\langle L\rangle$. If
$\phi\in L^*\setminus\{0\}$, then there exists $s_n\in S$ such that
$\phi(s_n)\ne0$, and consequently
\[
u(\phi)\ge2^{-n}|\delta_{s_n}(\phi)|>0.
\]
We prove that $u$ is quasi-interior. Fix
$f\in\mathrm{FBL}\langle L\rangle$ and $\varepsilon>0$.

Since $f$ is continuous on $L^*$ and $f(0)=0$, there is a neighbourhood
$V$ of $0$ in $L^*$ such that
\[
|f(\phi)|<\varepsilon/2\qquad(\phi\in V).
\]

Put $K=L^*\setminus V$. Then $K$ is compact and $0\notin K$. Since $u$ is
continuous and $u(\phi)>0$ for every $\phi\in K$,
\[
\delta:=\min_{\phi\in K}u(\phi)>0.
\]

Choose an integer $m$ such that $m\delta>\|f\|_\infty$ and set
\[
f_m=(f\wedge mu)\vee(-mu)\in\mathrm{FBL}\langle L\rangle_u.
\]

For $\phi\in K$, the inequality
$mu(\phi)\ge m\delta>\|f\|_\infty\ge |f(\phi)|$ implies
$f_m(\phi)=f(\phi)$.

For $\phi\in V$, the elementary inequality
$|(a\wedge b)\vee(-b)|\le |a|$ for $a\in\mathbb R$ and $b\ge0$ gives
$|f_m(\phi)|\le |f(\phi)|$. Hence
\[
|f_m(\phi)-f(\phi)|\le |f_m(\phi)|+|f(\phi)|<\varepsilon.
\]

Thus $\|f_m-f\|_\infty\le\varepsilon$. Since $\varepsilon>0$ was arbitrary,
$f$ belongs to the closure of $\mathrm{FBL}\langle L\rangle_u$. Therefore
$\overline{\mathrm{FBL}\langle L\rangle_u}
=\mathrm{FBL}\langle L\rangle$, so $u$ is a quasi-interior point.
\end{proof}
\section{Banach lattice homomorphisms on $\FBL L$}\label{sec:lattice homo FBL}

\subsection{Induced Banach lattice homomorphisms}

Given a lattice homomorphism $T:L\to M$ between distributive lattices, the
universal property yields a unique Banach lattice homomorphism
$\overline T:\FBL L\to\FBL M$ satisfying
$\overline T\circ\delta_L=\delta_M\circ T$. This is analogous to the
operator induced by a bounded linear map between Banach spaces in the theory
of $\mathrm{FBL}[E]$. We describe $\overline T$ as a composition operator
and discuss surjectivity.
\begin{theorem}
Let $T:L\to M$ be a lattice homomorphism. Then the induced extension
\[
\overline T:\mathrm{FBL}\langle L\rangle\to\mathrm{FBL}\langle M\rangle
\]
is given, for $f\in\mathrm{FBL}\langle L\rangle$, by
\[
\overline T(f)=f\circ T^*,
\]
where $T^*:M^*\to L^*$ is defined by $T^*(x^*)=x^*\circ T$.
\end{theorem}

\begin{proof}
Consider the composition operator induced by $T^*$:
\[
C_{T^*}f(x^*)=f(T^*x^*)
\qquad(f\in\mathrm{FBL}\langle L\rangle,\ x^*\in M^*).
\]
Then $C_{T^*}:\FBL L\to C(M^*)$ is a lattice homomorphism. Moreover, for
every $x\in L$,
\[
C_{T^*}(\delta_x)=\delta_{Tx},
\]
so the uniqueness part of the universal property gives
\[
\overline T=C_{T^*}.
\]
\end{proof}

\begin{remark}
If $\overline T$ is injective, then $T$ is injective. The converse does not
hold in general; see \cite[Example~4.4]{AMCRARZ2022b}.
\end{remark}

\begin{theorem}\label{thm:onto}
Let $T:L\to M$ be a lattice homomorphism. If $T$ is surjective, then the
induced Banach lattice homomorphism
\[
\overline T:\mathrm{FBL}\langle L\rangle\to\mathrm{FBL}\langle M\rangle
\]
is surjective.
\end{theorem}
\begin{proof} 
Let
\[
Z=\mathrm{FBL}\langle L\rangle/\ker(\overline T)
\]
and let $Q:\mathrm{FBL}\langle L\rangle\to Z$ be the quotient map. There is
a Banach lattice homomorphism $S:Z\to\mathrm{FBL}\langle M\rangle$ such that
\[
\overline T=SQ.
\]

Since $T$ is surjective, define $R:M\to Z$ by
\[
R(Tx)=Q\delta_L(x)\qquad(x\in L).
\]
This is well defined: if $T(x)=T(y)$, then
\[
\overline T(\delta_L(x))=\delta_M(Tx)=\delta_M(Ty)=\overline T(\delta_L(y)),
\]
so $\delta_L(x)-\delta_L(y)\in\ker(\overline T)$ and
$Q\delta_L(x)=Q\delta_L(y)$. Thus $Q\delta_L=RT$.

Let $\widehat R:\mathrm{FBL}\langle M\rangle\to Z$ be the canonical
extension of $R$. For $y\in M$, choose $x\in L$ with $y=Tx$. Then
\[
S\widehat R\delta_M(y)=SR(y)=SRT(x)=SQ\delta_L(x)
=\overline T\delta_L(x)=\delta_M(y).
\]
Thus $S\widehat R$ is the identity on the range of $\delta_M$, and hence on
the sublattice generated by it. Since this sublattice is dense in $\FBL M$,
$S\widehat R=\Id_{\FBL M}$. Therefore $S$ is surjective, and consequently
$\overline T=SQ$ is surjective.
\end{proof}

\begin{remark}
A modification of \cite[Example~4.4]{AMCRARZ2022b} shows that the converse to
Theorem~\ref{thm:onto} fails in general. Let
$L=\{0,a,b,1\}$ with $a\wedge b=0$ and $a\vee b=1$, and let
$M=\{0,a,1\}$ be a proper sublattice. The inclusion
$\iota:M\to L$ is not surjective. We claim, however, that
$\overline\iota:\FBL M\to\FBL L$ is onto. First observe that
$\iota^*:L^*\to M^*$, given by $\iota^*x^*=x^*\circ\iota$, is injective.
Indeed, for every $x^*\in L^*$ we have
$\{x^*(0),x^*(1)\}=\{x^*(a),x^*(b)\}$. If
$\iota^*x_1^*=\iota^*x_2^*$, then
    $$
\{x_1^*(a),x_1^*(b)\}
=\{x_1^*(0),x_1^*(1)\}
=\{x_2^*(0),x_2^*(1)\}
=\{x_2^*(a),x_2^*(b)\}.
$$
It follows that $x_1^*(b)=x_2^*(b)$, and hence $x_1^*=x_2^*$. Thus
$\iota^*$ is injective. Since $\overline\iota$ is the composition operator
associated with $\iota^*$, surjectivity of $\overline\iota$ follows.
\end{remark}

\begin{remark}
For a lattice homomorphism $T:M\to L$, Theorem~\ref{thm:onto} shows that
surjectivity of $\overline T$ is equivalent to surjectivity of
$\overline\iota$, where $\iota:T(M)\hookrightarrow L$ is the inclusion. If
$L$ has a maximum $M_L$ and a minimum $m_L$, with $m_L,M_L\in T(M)$, then
\cite[Theorem~2.7]{AMCRARZ2022a} and standard facts on composition operators between
$C(K)$ spaces show that surjectivity of $\overline T$ is equivalent to
injectivity of $\iota^*:K_L\to K_M$.
\end{remark}

\subsection{General lattice homomorphisms on $\FBL L$}

For free Banach lattices over Banach spaces, lattice homomorphisms
$\mathrm{FBL}[E]\to\mathrm{FBL}[F]$ can be represented through maps
$F^*\to E^*$ and a composition formula \cite[Section~10]{OTTT2024}. We now
record the analogous construction for $\FBL L$. Given a lattice homomorphism
$T:\FBL L\to\FBL M$, we associate a map $\Phi_T:M^*\to\mathbb R^L$ and
study its basic properties.

Let $L$ and $M$ be distributive lattices, and let
$T:\mathrm{FBL}\langle L\rangle\to\mathrm{FBL}\langle M\rangle$ be a
lattice homomorphism. For $x^*\in M^*$, define
$\Phi_T(x^*):L\to\mathbb R$ by
\begin{equation}\label{eq:phi_T}
\Phi_T(x^*)(x)=(T\delta_x)(x^*)\qquad(x\in L).
\end{equation}

\begin{proposition}
\label{prop:PhiT_algebraic}
Let $L$ and $M$ be lattices, and let
\[
T:\mathrm{FBL}\langle L\rangle\to\mathrm{FBL}\langle M\rangle
\]
be a lattice homomorphism. Let $\Phi_T:M^*\to\mathbb R^L$ be defined by
\eqref{eq:phi_T}. Then:
\begin{enumerate}
\item For every $x^*\in M^*$, the map $\Phi_T(x^*):L\to\mathbb R$ is a
lattice homomorphism.
\item For every $x^*\in M^*$ and every $x\in L$,
\[
|(\Phi_T(x^*))(x)|\le\|T\|.
\]
In particular, if $\|T\|\le1$, then $\Phi_T(M^*)\subset L^*$.
\item The map $\Phi_T$ is positively homogeneous: if $x^*\in M^*$ and
$\lambda\ge0$ are such that $\lambda x^*\in M^*$, then
\[
\Phi_T(\lambda x^*)=\lambda\Phi_T(x^*).
\]
\item If $\|T\|\le1$, then for every $f\in\FBL L$ and every $x^*\in M^*$,
\[
(Tf)(x^*)=f(\Phi_T(x^*)).
\]
\end{enumerate}
\end{proposition}

\begin{proof}
(1) Fix $x^*\in M^*$. For $x,y\in L$,
\[
(\Phi_T(x^*))(x\vee y)
=(T\delta_{x\vee y})(x^*)
=(T(\delta_x\vee\delta_y))(x^*)
=\max\{(T\delta_x)(x^*),(T\delta_y)(x^*)\}.
\]
Thus $\Phi_T(x^*)$ preserves joins. The proof for meets is analogous.

\medskip

(2) For $x^*\in M^*$ and $x\in L$,
\[
|(\Phi_T(x^*))(x)|=|(T\delta_x)(x^*)|
\le\|T\delta_x\|\le\|T\|.
\]

\medskip

(3) If $x^*\in M^*$ and $\lambda\ge0$ are such that
$\lambda x^*\in M^*$, then, for $x\in L$,
\[
(\Phi_T(\lambda x^*))(x)
=(T\delta_x)(\lambda x^*)
=\lambda(T\delta_x)(x^*)
=\lambda(\Phi_T(x^*))(x).
\]
(4) Assume that $\|T\|\le1$. Then $\Phi_T(x^*)\in L^*$ by (2). Let
$g\in\mathrm{FVL}\langle L\rangle$. We first prove that
\[
(Tg)(x^*)=g(\Phi_T(x^*))\qquad(x^*\in M^*).
\]

For $g=\delta_x$, this follows from
\[
(T\delta_x)(x^*)=(\Phi_T(x^*))(x)=\delta_x(\Phi_T(x^*)).
\]
Since both sides are preserved under linear combinations and lattice
operations, the identity holds for every $g\in\mathrm{FVL}\langle L\rangle$.

Let $f\in\mathrm{FBL}\langle L\rangle$, and choose
$(g_n)\subset\mathrm{FVL}\langle L\rangle$ with $g_n\to f$ in norm. Since
$T$ is bounded and evaluation at $x^*\in M^*$ is continuous,
\[
(Tf)(x^*)=\lim_{n\to\infty}(Tg_n)(x^*)
=\lim_{n\to\infty}g_n(\Phi_T(x^*)).
\]

Since $\Phi_T(x^*)\in L^*$ and evaluation at $\Phi_T(x^*)$ is continuous,
\[
\lim_{n\to\infty}g_n(\Phi_T(x^*))=f(\Phi_T(x^*)).
\]
This proves the formula.
\end{proof}

\begin{proposition}
\label{prop:PhiT_topological}
Let $L$ and $M$ be lattices, and let
\[
T:\mathrm{FBL}\langle L\rangle\to\mathrm{FBL}\langle M\rangle
\]
be a lattice homomorphism. Then the associated map
$\Phi_T:M^*\to\mathbb R^L$ has the following properties.
\begin{enumerate}
\item The map $\Phi_T$ is continuous when $M^*$ and $\mathbb R^L$ are
endowed with the product topologies.
\item If $T$ has dense range and $\|T\|\le1$, then $\Phi_T$ is injective.
\item If $T$ is an isometric lattice isomorphism, then $\Phi_T$ is bijective and
\[
\Phi_{T^{-1}}=(\Phi_T)^{-1}.
\]
\item For every $x_1^*,\dots,x_m^*\in M^*$,
\[
\sup_{y\in L}\sum_{i=1}^m|\Phi_T(x_i^*)(y)|
\le \|T\|\sup_{x\in M}\sum_{i=1}^m|x_i^*(x)|.
\]
\end{enumerate}
\end{proposition}

\begin{proof}
(1) For each $x\in L$, the coordinate map
$x^*\mapsto(\Phi_T(x^*))(x)=(T\delta_x)(x^*)$ is continuous on $M^*$, since
$T\delta_x\in\mathrm{FBL}\langle M\rangle\subset C(M^*)$.

\medskip

(2) Assume that $T$ has dense range and $\|T\|\le1$. Let
$x_1^*,x_2^*\in M^*$ with $x_1^*\ne x_2^*$. Since
$\mathrm{FBL}\langle M\rangle$ separates points of $M^*$, choose
$h\in\mathrm{FBL}\langle M\rangle$ with $h(x_1^*)\ne h(x_2^*)$. By density
of the range of $T$, there is $f\in\mathrm{FBL}\langle L\rangle$ such that
$(Tf)(x_1^*)\ne(Tf)(x_2^*)$. Proposition~\ref{prop:PhiT_algebraic} gives
\[
f(\Phi_T(x_1^*))\ne f(\Phi_T(x_2^*)),
\]
so $\Phi_T(x_1^*)\ne\Phi_T(x_2^*)$.

\medskip

(3) Since $\|T\|=\|T^{-1}\|=1$, Proposition~\ref{prop:PhiT_algebraic}
shows that $\Phi_T(M^*)\subset L^*$ and
$\Phi_{T^{-1}}(L^*)\subset M^*$.

For $x^*\in M^*$ and $y\in M$, Proposition~\ref{prop:PhiT_algebraic}
applied to $T^{-1}$ gives
\[
(T^{-1}\delta_y)(\Phi_T(x^*))
=\delta_y(\Phi_{T^{-1}}(\Phi_T(x^*))).
\]
Since $T(T^{-1}\delta_y)=\delta_y$, the left-hand side equals
$\delta_y(x^*)=x^*(y)$. Hence
$\Phi_{T^{-1}}\circ\Phi_T=\Id_{M^*}$.

The reverse identity is obtained in the same way, so $\Phi_T$ is bijective
with inverse $\Phi_{T^{-1}}$.

(4) Fix $y\in L$. By the definition of $\Phi_T$,
\[
\sum_{i=1}^m|\Phi_T(x_i^*)(y)|
=\sum_{i=1}^m|(T\delta_y)(x_i^*)|.
\]
The definition of the free norm gives
\[
\sum_{i=1}^m|(T\delta_y)(x_i^*)|
\le \|T\delta_y\|\sup_{x\in M}\sum_{i=1}^m|x_i^*(x)|.
\]
Since $\|T\delta_y\|\le\|T\|$, it follows that
\[
\sum_{i=1}^m|\Phi_T(x_i^*)(y)|
\le\|T\|\sup_{x\in M}\sum_{i=1}^m|x_i^*(x)|.
\]
Taking the supremum over $y\in L$ yields the result.
\end{proof}

\section{Lattice isometries on $\FBL L$}

If two lattices $L$ and $M$ are lattice isomorphic, then $\FBL L$ and
$\FBL M$ are Banach lattice isometric. The converse does not hold in
general. Recall that the \emph{opposite lattice} $L^{\mathrm{op}}$ has the
same underlying set as $L$, with the order reversed: $x\le y$ in
$L^{\mathrm{op}}$ precisely when $y\le x$ in $L$. Thus
\[
x\vee_{L^{\mathrm{op}}}y=x\wedge_L y,
\qquad
x\wedge_{L^{\mathrm{op}}}y=x\vee_L y.
\]

\begin{lemma}\label{lem:dual-correspondence}
Let $L$ be a lattice and let $L^{\mathrm{op}}$ be its opposite lattice.
Then
\[
x^*\in L^*\quad\text{if and only if}\quad -x^*\in(L^{\mathrm{op}})^*.
\]
Consequently, the map $x^*\mapsto -x^*$ is a bijection from $L^*$ onto
$(L^{\mathrm{op}})^*$.
\end{lemma}

\begin{proof}
Suppose that $x^*\in L^*$. We show that $-x^*:L^{\mathrm{op}}\to[-1,1]$
is a lattice homomorphism. If $a\le_{L^{\mathrm{op}}}b$, then
$b\le_L a$, so $x^*(b)\le x^*(a)$ and
\[
-x^*(a)\le -x^*(b).
\]
Thus $-x^*$ is order preserving on $L^{\mathrm{op}}$.

For joins in $L^{\mathrm{op}}$,
\begin{align*}
(-x^*)(a\vee_{L^{\mathrm{op}}}b)
&=-x^*(a\wedge_L b)\\
&=-\min\{x^*(a),x^*(b)\}\\
&=\max\{-x^*(a),-x^*(b)\}\\
&=(-x^*)(a)\vee(-x^*)(b).
\end{align*}

Similarly, for meets in $L^{\mathrm{op}}$,
\begin{align*}
(-x^*)(a\wedge_{L^{\mathrm{op}}}b)
&=-x^*(a\vee_L b)\\
&=-\max\{x^*(a),x^*(b)\}\\
&=\min\{-x^*(a),-x^*(b)\}\\
&=(-x^*)(a)\wedge(-x^*)(b).
\end{align*}

Thus $-x^*\in(L^{\mathrm{op}})^*$. The converse follows by applying the same
argument to $L^{\mathrm{op}}$. The map $x^*\mapsto -x^*$ is its own inverse,
so it is a bijection.
\end{proof}
\begin{proposition}\label{prop:dual-isometry}
For every lattice $L$, the Banach lattices $\FBL L$ and
$\FBL{L^{\mathrm{op}}}$ are lattice isometric.
\end{proposition}

\begin{proof}
Let $S:(L^{\mathrm{op}})^*\to L^*$ be given by $S(x^*)=-x^*$, and define
\[
\Phi f=f\circ S\qquad(f\in\FBL L).
\]
By Lemma~\ref{lem:dual-correspondence}, $S$ is a bijection. It is immediate
that $\Phi$ preserves linear and lattice operations. Moreover, the defining
norm formula for free Banach lattices is invariant under the change of
variables $x^*\mapsto -x^*$, and therefore $\|\Phi f\|=\|f\|$. Since
$\Phi^{-1}$ is defined in the same way, $\Phi$ is a lattice isometry.
\end{proof}
\begin{remark}
There are lattices that are not isomorphic to their opposite lattices, while
their free Banach lattices are lattice isometric by
Proposition~\ref{prop:dual-isometry}. For instance, one may take a five-point
lattice $L=\{0,1,a,b,c\}$ with $0<a=b\wedge c<1$ and $b\vee c=1$.
\end{remark}
Motivated by \cite[\S~10.3]{OTTT2024}, it is natural to ask for conditions on a
lattice $L$ under which lattice isometry of $\FBL M$ and $\FBL L$ forces
$M$ to be lattice isomorphic to either $L$ or $L^{\mathrm{op}}$.

\section*{Acknowledgements}

This work was supported by the CSIC cooperation grant COOPB25033 under the
i-COOP program. P. Tradacete was also partially supported by grants
PID2024-162214NB-I00 and CEX2023-001347-S, funded by Agencia Estatal de
Investigaci\'on (Spain) MCIN/AEI/10.13039/501100011033.

\end{document}